\documentclass[12pt,a4paper]{amsart}
\usepackage{geometry}
\usepackage{amsmath,amssymb,amsfonts,amscd,amsthm,wasysym,color,enumitem,indentfirst,graphicx,booktabs}
\usepackage[bookmarksnumbered,colorlinks,linktocpage,plainpages]{hyperref}
\hypersetup{pdfstartview={FitH}}

\newtheorem{thmx}{Theorem}

\numberwithin{equation}{section}
\newtheorem{theorem}{Theorem}[section]

\theoremstyle{definition}

\newtheorem{question}{Question}

\newtheorem{definition}{Definition}

\allowdisplaybreaks[4]

\makeatletter
\@namedef{subjclassname@2020}{\textup{2020} Mathematics Subject Classification}
\makeatother

\begin{document}
\title[Bott-Chern cohomology and the Hartogs extension theorem]{Bott-Chern cohomology and the Hartogs extension theorem for pluriharmonic functions}
%\date{May 5, 2022}

\author{Xieping Wang}
\address{School of Mathematical Sciences, University of Science and Technology of China, Hefei 230026, Anhui, People's Republic of China}
\email{xpwang008@ustc.edu.cn}

\thanks{The author was partially supported by NSFC grant 12001513, and NSF of Anhui Province grants 2008085QA18 and 2108085QA04.}


\dedicatory{To Li and Rui'an}

\begin{abstract}
Let $X$ be a cohomologically $(n-1)$-complete complex manifold of dimension $n\geq 2$. We prove a vanishing result for the Bott-Chern cohomology group of type $(1, 1)$  with compact support in $X$, which combined with the well-known technique of Ehrenpreis implies a Hartogs type extension theorem for pluriharmonic functions on $X$.
\end{abstract}
\maketitle

\section{Introduction}
As is well known, the Hartogs extension theorem for holomorphic functions is one of the striking results in complex analysis that distinguish the higher-dimensional case from the one-dimensional case. In its general form, the theorem states:

\begin{thmx}\label{thm:Hartogs_Holo}
Let $X$ be a complex manifold of dimension $n\geq 2$ which is  $(n-1)$-complete in the sense of Andreotti-Grauert. Suppose $\Omega$ is a domain in $X$ and $K$ is a compact subset of $\Omega$ such that $\Omega\!\setminus\!K$ is connected. Then every holomorphic function on $\Omega\!\setminus\!K$ has a unique holomorphic extension to $\Omega$.
\end{thmx}

The theorem was proved by Andreotti-Hill \cite{Andreotti-Hill} in 1972 and generalized to $(n-1)$-complete normal complex spaces by Merker-Porten \cite{Merker-Porten} in 2009; see \cite{CoRuppen} or \cite{OV-Trans} for a simple alternative proof.  Col\c{t}oiu-Ruppenthal \cite{CoRuppen} also dealt with the more general cohomologically $(n-1)$-complete case. We refer the reader to these works and the references therein for a more detailed account of Hartogs' extension theorem, and to \cite{Vijiitu} for recent developments on related topics. It should also be mentioned that there is a metric analogue of Hartogs' theorem under appropriate geometric conditions; see \cite{Gau-Zimmer} for details.

On the other hand, Chen \cite{Chen} recently proved the following very interesting result.

\begin{thmx}\label{Thm:Hartogs_PH}
Let $\Omega$ be a domain in ${\mathbb C}^n$, $n\geq 2$, and let $K$ be a compact subset of $\Omega$ such that  $\Omega\!\setminus\!K$ is connected. Then every pluriharmonic function on $\Omega\!\setminus\!K$ has a unique pluriharmonic extension to $\Omega$.
\end{thmx}

Inspired by this result and the work mentioned above, it seems natural to ask the following

\begin{question}\label{Q:Hartogs_extension}
Does Theorem \ref{Thm:Hartogs_PH} still hold when $\mathbb C^n$ is replaced by a Stein (or more generally, cohomologically $(n-1)$-complete) manifold of dimension $n\geq 2$?
\end{question}

Now recall that the Bott-Chern cohomology group of type $(1, 1)$  with compact support in a complex manifold $X$ is defined as
   $$
   H^{1,\, 1}_{{\rm BC},\, c}(X)=\frac{{\rm Ker}\, d\cap \mathcal{D}^{1,\, 1}(X)}{\partial\bar{\partial}\,\mathcal{D}(X)},
   $$
where $\mathcal{D}(X)$ (resp. $\mathcal{D}^{1,\, 1}(X)$) denotes the set of all compactly supported smooth functions (resp. $(1, 1)$-forms) on $X$. An equivalent formulation of Theorem \ref{Thm:Hartogs_PH} is $H^{1,\, 1}_{{\rm BC},\, c}({\mathbb C}^n)=0$ for $n\geq 2$. Chen proved this cohomology-vanishing result by directly solving the $\partial\bar{\partial}$-equation, which depends heavily on an elegant trick of Lelong \cite{Lelong} and the Sobolev inequality  (see \cite[Section 3]{Chen} for details). According to the celebrated work of Mok-Siu-Yau \cite{Mok-Siu-Yau} and Croke \cite{Croke}, these two key ingredients apply only to complete K\"{a}hler manifolds of {\it nonnegative} holomorphic bisectional curvature.\footnote{Despite this, Chen's method is attractive partly because it yields the desired solution in an efficient way, with a precise gradient estimate.} For the question posed earlier, these seem to constitute only a rather restrictive class of complex manifolds, since by virtue of the well-known embedding theorem and the curvature-decreasing property for K\"{a}hler submanifolds, Stein manifolds--the objects of interest here--are usually those of {\it nonpositive} holomorphic bisectional curvature. In such a situation, to answer the previous question we instead use a (qualitative) cohomological argument and prove the following

\begin{theorem}\label{thm:Hartogs extension}
Let $X$ be a cohomologically $(n-1)$-complete complex manifold of dimension $n\geq 2$. Suppose $\Omega$ is a domain in $X$ and $K$ is a compact subset of $\Omega$ such that $\Omega\!\setminus\!K$ is connected. Then every pluriharmonic function on $\Omega\!\setminus\!K$ has a unique pluriharmonic extension to $\Omega$.
\end{theorem}

By adapting Ehrenpreis' technique for the $\bar{\partial}$-operator (see \cite{Ehrenpreis, HormanderBook}) to the case of the $\partial\bar{\partial}$-operator, we reduce Theorem \ref{thm:Hartogs extension} to the following cohomology-vanishing result.

\begin{theorem}\label{thm:BC-vanishing}
Let $X$ be a cohomologically $(n-1)$-complete complex manifold of dimension $n\geq 2$. Then
   $$
   H^{1,\, 1}_{{\rm BC},\, c}(X)=0.
   $$
\end{theorem}
This follows easily from the Serre duality \cite{Serre}, as we will show in Sect. \ref{sect:result-proof}.
If $X$ is Stein, the result is previously known; see \cite[Teorema 2.4]{Bigolin}.

We now conclude the introduction with some remarks on Theorem \ref{thm:Hartogs extension}. Using the unique continuation property for holomorphic forms, one can easily show that Theorem \ref{thm:Hartogs extension} implies the Hartogs extension theorem for holomorphic functions. And as essentially pointed out in \cite{Merker-Porten}, the latter does not hold in general if the cohomological $(n-1)$-completeness assumption of $X$ is weakened to (cohomological) $n$-completeness, which is equivalent to requiring that $X$ have no compact connected components according to  the Serre duality and a result of Greene-Wu (see \cite[Chapter IX, Theorem 3.5]{Demaillybook}). These show that the cohomological completeness condition imposed on $X$ in Theorem \ref{thm:Hartogs extension} is sharp.

\medskip
\noindent {\bf Acknowledgements.}
The author is grateful to the anonymous referees for  their careful reading of the paper and helpful comments.

\section{Proof of the results}\label{sect:result-proof}
We begin by recalling the notion of cohomological $q$-completeness, where $q$ is a positive integer.

\begin{definition}
A complex manifold $X$ is called {\it cohomologically $q$-complete} if its $k$-th cohomology group with coefficients in $\mathcal{F}$, $H^k(X,\, \mathcal{F})$, vanishes for every coherent analytic sheaf $\mathcal{F}$ over $X$ and integer $k\geq q$.
\end{definition}

Recall also that an $n$-dimensional complex manifold is called {\it $q$-complete} in the sense of Andreotti-Grauert \cite{Andreotti-Grauert} if it possesses a smooth exhaustion function whose Levi form has at least $n-q+1$ positive eigenvalues at every point. According to Grauert's solution of the Levi problem (see, e.g., \cite{HormanderBook}), $1$-complete complex manifolds are precisely Stein manifolds. In general, by an important result of Andreotti-Grauert \cite{Andreotti-Grauert} we know that $q$-complete complex manifolds are necessarily cohomologically $q$-complete, but the converse remains unclear except for the cases $q=1$ and $q\geq {\rm dim}\,X$.

\smallskip
We are now ready to prove the results of this paper, starting with Theorem \ref{thm:BC-vanishing}.

\begin{proof}[Proof of Theorem $\ref{thm:BC-vanishing}$]
Without loss of generality, we may assume that $X$ is connected. Let $\Omega^p_X$ denote the sheaf of germs of holomorphic $p$-forms on $X$, $p=0,\ldots,n$. It is well-known that $\Omega^p_X$ is coherent. Since $n\geq 2$, the cohomological $(n-1)$-completeness of $X$ implies
 $$
 H^q(X,\, \Omega^p_X)=0 \quad {\rm for}\ \,\, p,\, q=n-1,\,n.
 $$
Then by Serre's criterion \cite[Proposition 6]{Serre}, we can invoke the Serre duality \cite[Th\'{e}or\`{e}me 2]{Serre} to get
 $$
 H^1_c(X,\, \Omega^p_X)=0\quad {\rm for}\ \,\, p=0,\,1.
 $$
Equivalently, we have the corresponding vanishing result for the Dolbeault cohomology groups of $X$ with compact support:
 $$
 H^{0,\, 1}_{\bar{\partial},\, c}(X)=0
 \quad {\rm and}\quad H^{1,\, 1}_{\bar{\partial},\, c}(X)=0,
 $$
in view of the Dolbeault isomorphism theorem (with compact support). Furthermore, using $H^n(X,\, \Omega^n_X)=0$ and the Serre duality again  we obtain
 $$
 H^0_c(X,\, \Omega^0_X)=0.
 $$
This means that $X$ is noncompact.

We now show that what we have obtained in turn implies
   $$
   H^{1,\, 1}_{{\rm BC},\, c}(X)=0.
   $$
To see this, let $\mathcal{D}^{p,\, q}(X)$ denote the set of all compactly supported  smooth  $(p, q)$-forms on $X$. Then for every $d$-closed form $\alpha^{1,\, 1}\in\mathcal{D}^{1,\, 1}(X)$,  the vanishing of $H^{1,\, 1}_{\bar{\partial},\, c}(X)$ implies that there exists a form
$\alpha^{1,\, 0}\in \mathcal{D}^{1,\, 0}(X)$ such that $\bar{\partial}\alpha^{1,\, 0}=\alpha^{1,\, 1}$.
Now
  $$
  \bar{\partial}\partial\alpha^{1,\, 0}=-\partial\bar{\partial}\alpha^{1,\, 0}=-\partial\alpha^{1,\, 1}=0.
  $$
It follows that $\partial\alpha^{1,\, 0}$ is a compactly supported holomorphic $2$-form  on the noncompact manifold $X$, hence vanishes identically. Since $H^{0,\, 1}_{\bar{\partial},\, c}(X)=0$, then there exists a function $\alpha\in \mathcal{D}(X)$ such that $\partial\alpha=-\alpha^{1,\, 0}$. Consequently
  $$
  \partial\bar{\partial}\alpha=\bar{\partial}\alpha^{1,\, 0}=\alpha^{1,\, 1},
  $$
and we are done.
\end{proof}

\begin{proof}[Proof of Theorem $\ref{thm:Hartogs extension}$]
Choose a cut-off function $\chi\in \mathcal{D}(\Omega)$ such that $\chi=1$ on a neighborhood of $K$. Given a pluriharmonic function $f$ on $\Omega\!\setminus\!K$, we define
  $$
  v:=\partial\bar{\partial}((1-\chi)f).
  $$
Then $v$ is a smooth $d$-closed $(1, 1)$-form on $X$  satisfying
  $$
  {\rm supp}\,v\subset {\rm supp}\,\chi\subset\Omega.
  $$
By Theorem \ref{thm:BC-vanishing}, there exists a function $u\in\mathcal{D}(X)$ such that $\partial\bar{\partial} u=v$.

Now set
  $$
  F:=(1-\chi)f-u,
  $$
which is pluriharmonic on $\Omega$. We claim that $\left.F\right|_{\Omega\setminus K}=f$. For this we may assume that $X$ itself is connected, since otherwise $\Omega$ would be contained in a single connected component of $X$ and we only need to replace $X$ with this connected component. Note that $u$ is  pluriharmonic (and hence real-analytic) on $X\!\setminus\!{\rm supp}\,\chi$, which intersects $X\!\setminus\!{\rm supp}\,u$ since $X$ is noncompact (as we pointed out in the proof of Theorem \ref{thm:BC-vanishing}). Then by the unique continuation property, $u$ vanishes identically on some connected component, say $U$, of $X\!\setminus\!{\rm supp}\,\chi$. On the other hand, the connectedness of $X$ implies that $U$ has a nonempty boundary in $X$, which is necessarily contained in ${\rm supp}\,\chi\subset\Omega$ (by the fact that $U$ is a connected component of $X\!\setminus\!{\rm supp}\,\chi$). We then conclude that
  $$
  U\cap(\Omega\!\setminus\!{\rm supp}\,\chi)=U\cap\Omega\neq \emptyset
  $$
and
  $$
  \left.F\right|_{U\cap(\Omega\setminus {\rm supp}\,\chi)}
  =\left.f\right|_{U\cap(\Omega\setminus {\rm supp}\,\chi)}.
  $$
Consequently, $\left.F\right|_{\Omega\setminus K}=f$ by the connectedness of $\Omega\!\setminus\!K$.
\end{proof}


\begin{thebibliography}{99}
\bibitem{Andreotti-Grauert} A. Andreotti and  H. Grauert, {\it Th\'{e}or\`{e}me de finitude pour la cohomologie des espaces complexes}, Bull. Soc. Math. France {\bf 90} (1962), 193--259.

\bibitem{Andreotti-Hill} A. Andreotti and C. Denson Hill, {\it E. E. Levi convexity and the Hans Lewy problem}, I and II, Ann. Scuola Norm. Sup. Pisa Cl. Sci. (3) {\bf 26} (1972), 325--363 and 747--806.

\bibitem{Bigolin} B. Bigolin, {\it Gruppi di Aeppli}, Ann. Scuola Norm. Sup. Pisa Cl. Sci. (3) {\bf 23} (1969), 259--287.

%\bibitem[BC]{Bott-Chern} R. Bott and S. S. Chern, {\it Hermitian vector bundles and the equidistribution of the zeroes of their holomorphic sections}, Acta Math. {\bf 114} (1965), 71--112.

\bibitem{Chen} B.-Y. Chen, {\it  Hardy-Sobolev type inequalities and their applications}, arXiv:1712.02044v2.

\bibitem{CoRuppen} M. Col\c{t}oiu and J. Ruppenthal, {\it  On Hartogs' extension theorem on $(n-1)$-complete complex spaces}, J. Reine Angew. Math. {\bf 637} (2009), 41--47.

\bibitem{Croke} C. B. Croke, {\it Some isoperimetric inequalities and eigenvalue estimates}, Ann. Sci. \'{E}cole Norm. Sup. {\bf 13} (1980), 419--435.

\bibitem{Demaillybook} J.-P. Demailly, ``Complex Analytic and Differential Geometry", available  at \url{https://www-fourier.ujf-grenoble.fr/~demailly/books.html}.

\bibitem{Ehrenpreis} L. Ehrenpreis, {\it A new proof and an extension of Hartogs' theorem}, Bull. Amer. Math. Soc. {\bf 67} (1961), 507--509.


\bibitem{Gau-Zimmer} H. Gaussier and A. Zimmer,  {\it A metric analogue of Hartogs' theorem}, arXiv:2111.12029v2.

\bibitem{HormanderBook}  L. H\"ormander, ``An Introduction to Complex Analysis in Several Variables", 3rd ed., North-Holland Mathematical Library, vol. 7, North-Holland Publishing Co., Amsterdam, 1990.

\bibitem{Lelong} P. Lelong, {\it Fonctions entie\`{e}res ($n$ variables) et fonctions plurisousharmoniques d'ordre fini dans ${\mathbb C}^n$}, J. Anal. Math. {\bf 12} (1964), 365--407.



\bibitem{Merker-Porten} J. Merker and E. Porten, {\it The Hartogs extension theorem on $(n-1)$-complete complex spaces}, J. Reine Angew. Math. {\bf 637} (2009), 23--39.

\bibitem{Mok-Siu-Yau} N. Mok, Y.-T. Siu and S.-T. Yau, {\it The Poincar\'e-Lelong equation on complete K\"{a}hler manifolds}, Compositio Math. {\bf 44} (1981), 183--218.



\bibitem{OV-Trans} N. {\O}vrelid and  S. Vassiliadou, {\it Semiglobal results for $\bar{\partial}$ on complex spaces with arbitrary singularities, Part II}, Trans. Amer. Math. Soc. {\bf 363} (2011), 6177--6196.

\bibitem{Serre} J.-P. Serre, {\it Un th\'{e}or\`{e}me de dualit\'{e}}, Comment. Math. Helv. {\bf 29} (1955), 9--26.


\bibitem{Vijiitu} V. V\^{i}j\^{i}itu, {\it On Hartogs' extension}, Ann. Mat. Pura Appl. {\bf 201} (2022),  487--498.


\end{thebibliography}
\end{document}